\documentclass[a4paper,12pt,reqno]{amsart}
\usepackage{amsmath,amsthm,mathrsfs}
\usepackage{latexsym}
\usepackage{amssymb}
\usepackage[all]{xy}
\usepackage{enumerate}
\usepackage{pst-all}
\usepackage{pstricks}
\usepackage{enumerate}
\usepackage{lipsum}
\usepackage{float}
\usepackage{multicol}
\input{xy}
\theoremstyle{definition}
 \newtheorem{theo}{Theorem}[section]
 \newtheorem{propo}[theo]{Proposition}
 \newtheorem{lem}[theo]{Lemma}
 \newtheorem{cor}[theo]{Corollary}
 \newtheorem{obs}[theo]{Observation}
 \newtheorem{ejem}[theo]{Example}
 \newtheorem{defin}[theo]{Definition}
\def\L{\Lambda}
\theoremstyle{definition}

\newcolumntype{C}{>{$}c<{$}}
\newcolumntype{L}{>{$}l<{$}}
\newcolumntype{R}{>{$}r<{$}}
\def\G{\Gamma}
\def\oo{\mathfrak{o}}

\def\C{\mathscr{C}}

\def\T{\mathscr{T}}
\def\V{\mathscr{V}}
\def\Q{\mathcal{Q}}
\def\CM{\mathcal{C}}

\def\occ{\textrm{occ}}
\def\val{\textrm{val}}
\def\CC{\mathscr{CC}}
\def\hom{\textrm{Hom}}

\def\rad{\textrm{rad}}

\def\dime{\textrm{dim}}
\def\gcd{\textrm{gcd}}
\def\Z{\mathbb{Z}}
\def\GG{\mathscr{G}}
\newcommand{\ov}{\overline}

\makeatletter
\@namedef{subjclassname@2020}{\textup{2020} Mathematics Subject Classification}
\makeatother

\title[BCA's induced by finite groups]{On Brauer configuration algebras induced by finite groups}
\author{Alex Sierra C.}
\date{}
\keywords{Brauer configuration algebra, Cartan matrix, length of a module, finite groups, subgroup-occurrence of an element}

\address{Faculdade de Matem\'atica, Universidade Federal do Par\'a, Guam\'a, CEP 66075-110, Bel\'em - PA, Brasil}

\email{alexsc@ufpa.br\\alhadamard@gmail.com}

\subjclass[2020]{16G20, 16G99, 20D60}

\thanks{\underline{This is a preliminary version}}

\begin{document}
 
 \begin{abstract}
 In this article we calculate two aspects of the representation theory of a Brauer configuration algebra: its Cartan matrix, and the module length of its associated indecomposable projective modules. Then we introduce the concept of {\em subgroup-occurrence} of an element in a group and use the previous aspects to demonstrate combinatorial equalities satisfied for any finite group.

 \end{abstract}
\maketitle 
 
\section*{Introduction}

Brauer configuration algebras are a new class of associative symmetric finite dimensional algebras that were introduced in the mathematical literature in 2017 by E. Green and S. Schroll in \cite{brau}. Initially they were defined as a generalization of Brauer graph algebras, in the sense that every Brauer graph algebra is a Brauer configuration algebra. Since their introduction, a few aspects of these algebras have been studied, and in some cases completely determined. For example, the vector dimension of both a Brauer configuration algebra \cite[Proposition 3.13]{brau} and the center of a Brauer configuration algebra \cite[Theorem 4.9]{mya} have been successfully calculated. Recently, it was shown that it is possible to associate a Brauer configuration algebra to a dessin d'enfant, and also that the vector dimension of the associated Brauer configuration algebra, as well as the vector dimension of its center, are invariant under some natural action of the Galois group Gal$(\overline{\mathbb{Q}}/\mathbb{Q})$ \cite{MalicSchroll}.\\

The present work is divided mainly in two parts. The first part is directed to the determination of two aspects of the representation theory of a Brauer configuration algebra: its Cartan matrix and the module length of its associated indecomposable projective modules. The second part is directed to the introduction of a new technique that uses the Brauer configuration algebras as a tool. The general idea is as follows. It will be shown that using the aspects of the representation theory determined in the first part, we can extract combinatorial information satisfied by structures or objects that, in certain form, are susceptible to be encoded in terms of a class made up of Brauer configurations, obtaining in this way results that were unknown until now. In this paper, this technique is applied for the class of finite groups of an order different from a prime number, because, as it will be seen, a finite group of this type induces a family of Brauer configurations, i.e, it can be encoded in terms of a class of Brauer configurations. After obtaining these results satisfied for the class of finite groups of order different from a prime number, it is possible to extend them to the class of all finite groups. One of the reasons for achieving these results, is due to the fact that the concept of the subgroup-occurrence of an element in a group is used, which may be interpreted as the valence function of a Brauer configuration in the context of the Brauer configuration induced by a finite group. Few  other properties that follow from the definition of the subgroup-occurrence are studied in this paper. This concept seems promising.\\

The paper is outlined as follows. We start Section \ref{sec02} by giving the basic preliminaries about Brauer configuration algebras, such as definitions and examples. In Section \ref{sec3} the necessary definitions and the technique to determine the vector dimension of the $K$-spaces $v\L w$ are exposed, and the values of these dimensions as well. 
In Section \ref{sec5} the explicit form of the Cartan matrix of a Brauer configuration algebra is presented and a few examples are considered. In Section \ref{sec6} non-projective uniserial modules over a Brauer configuration algebra are presented and described in the same way as they appear in  \cite[Section 3.4]{brau}, and then they are used for the  construction of a composition series of a class of submodules contained in an indecomposable projective module associated to a Brauer configuration algebra.  Then, this composition series is used to calculate explicitly the length formula of any indecomposable projective module associated to a Brauer configuration algebra. In Section \ref{sec_indu} the concept of subgroup-occurrence of an element in a group is introduced and a few properties that follow from this concept are proved. After this, it is demonstrated that any finite group of order different from a prime number induces a family of Brauer configurations. Then, by applying the results obtained in Section \ref{sec5}  and Section \ref{sec6}, new combinatorial properties about finite groups are demonstrated (Theorem \ref{057} and Theorem \ref{062}). In Section \ref{sec_center}, it is shown that given a pair  $(G,\mu)$, where $G$ is a finite group with identity element $e$ and $\mu:G\to\Z_{>0}$ is a function such that $\mu(e)=1$, a class formed by associative finite dimension $K$-algebras can be associated to it in such a way that when $\mu\equiv1$, then the vector dimension of the center of any of these $K$-algebras coincides with the number of subgroups contained in $G$ (Theorem \ref{066}). 

In some of the sections it is assumed that the base field $K$ is algebraically closed.

\section{Brauer configuration algebras}\label{sec02}
In this section we present the definition of a Brauer configuration algebra. We also set the notation and terminology that  will be used in most of the nexts sections.\\

A \textit{Brauer configuration} is a quadruple $\G=(\G_0,\G_1,\mu,\oo)$ given by:
\begin{enumerate}[(1)]
\item $\G_0$ is a finite set of elements that we call \textit{vertices};
\item $\G_1$ is a finite collection of finite labeled multisets\footnote{A multiset is a set where repetitions of elements are allowed.} each of them formed by vertices. We call each element of $\G_1$ a \textit{polygon};
\item $\mu:\G_0\to\mathbb{Z}_{>0}$ is a set function that we call the \textit{multiplicity function};
\item for $\alpha\in\G_0$ and $V\in\G_1$ we denote by $\occ(\alpha,V)$ the number of times that the vertex $\alpha$ occurs in $V$. By $\val(\alpha)$ we denote the  integer value \[\val(\alpha):=\sum_{V\in\G_1}\occ(\alpha,V).\] A vertex $\alpha\in\G_0$ is called \textit{truncated} if $\val(\alpha)=\mu(\alpha)=1$. The \textit{orientation} $\oo$ means that for each nontruncated vertex $\alpha$ there is a chosen cyclic ordering of the polygons that contain $\alpha$, including repetitions. (See Observation \ref{003})
\end{enumerate}

Additionally to this we require that $\G$ satisfies the following conditions (see \cite[Definition 1.5]{brau}).
 \begin{enumerate}
  \item[C1.] Every vertex in $\G_0$ is a vertex in at least one polygon in $\G_1$.
  \item[C2.] Every polygon in $\G_1$ has at least two vertices.
  \item[C3.] Every polygon in $\G_1$ has at least one vertex $\alpha$ such that $\val(\alpha)\mu(\alpha)>1$.
 \end{enumerate}

\begin{obs}\label{003}
For each $\alpha\in\G_0$ such that $\val(\alpha)=t>1$ or $\mu(\alpha)>1$, let $V_1,\ldots, V_t$ be the list of polygons in which $\alpha$ occurs as a vertex, and where a polygon $V$ occurs $\occ(\alpha,V)$ times in the list, that is $V$ occurs the same number of times that $\alpha$ occurs as a vertex in $V$. The cyclic order at vertex $\alpha$ is obtained by linearly ordering the list, say $V_{i_1}<\cdots<V_{i_t}$ and by adding $V_{i_t}<V_{i_1}$. We observe that any cyclic permutation of a chosen cyclic ordering at vertex $\alpha$ can represent the same ordering. That is, if $V_1<\cdots<V_t$ is the chosen cyclic ordering at vertex $\alpha$, so is a cyclic permutation such as $V_2<V_3<\cdots<V_t<V_1$ or $V_3<V_4<\cdots<V_t<V_1<V_2$. If $V_{i_1}<\cdots<V_{i_t}$ is a cyclic order at the vertex $\alpha$ we denote it by $\alpha:V_{i_1}<\cdots<V_{i_t}$ and we call it a \textit{successor sequence at} $\alpha$. Now, if we have that $\mu(\alpha)>1$ but $\val(\alpha)=1$ and $V$ is the only polygon where $\alpha$ belongs, we denote the successor sequence at $\alpha$ simply by $\alpha:V$.
\end{obs}
\begin{ejem}\label{002}
Let $\G=(\G_0,\G_1,\mu,\oo)$ be the Brauer configuration given by the following data: $\G_0=\{1,2,3,4\},\G_1=\{V_1,V_2,V_3,V_4\}$ where $V_1=\{1,2\},V_2=\{1,2\},V_3=\{1,1,3,3\}$ and $V_4=\{3,4\}$. Observe that both $V_1$ and $V_2$ have the same set of vertices, however they are considered as different polygons in the configuration. If we define the multiplicity function as $\mu(3)=\mu(4)=1$ and $\mu(1)=\mu(2)=2$, we see that the vertex 4 is truncated. For the orientation $\oo$ we chose the following successor sequences.\[\begin{array}{rcl}1 & : & V_1<V_2<V_3<V_3;\\2 & : & V_1<V_2;\\3 & : & V_3<V_4<V_3.\end{array}\] Also observe that in this orientation, and any other chosen orientation, the number of times that a polygon appears in a successor sequence must coincide with the number of times that the associated vertex appears in the polygon.
\end{ejem}
For $\alpha\in\G_0$ a nontruncated vertex with $\val(\alpha)>1$, let $\alpha:V_{i_1}<\cdots<V_{i_t}$ be a successor sequence of $\alpha$, where $t=\val(\alpha)$. We say that $V_{i_{j+1}}$ is the \textit{successor} of $V_{i_j}$, for all $1\le j\le t$, and $V_{j_{t+1}}=V_{i_1}$. If $\val(\alpha)=1$ but $\mu(\alpha)>1$, and $V$ is the only polygon where $\alpha$ belongs, then we say that $V$ is its own successor at $\alpha$. We can also have that a polygon is its own successor at a vertex $\alpha$ where $\val(\alpha)>1$. This is the case for the polygon $V_3$ in the Example \ref{002} above. It is its own successor at the vertex 1.

Now, we define the induced quiver by a Brauer configuration. For $\G$ a Brauer configuration let $\Q$ be the quiver induced by
\begin{itemize}
\item The set of vertices of $\Q$ is in one-to-one correspondence with the set of polygons of $\G$. If $V$ is a polygon in $\G_1$, we will denote the associated vertex in $\Q$ by $v$, and we say that $v$ is the vertex in $\Q$ associated to $V$.
\item If the polygon $V'$ is a successor of the polygon $V$ at $\alpha$, then there is an arrow from $v$ to $v'$, where $v$ is the vertex in $\Q$ associated to $V$, and $v'$ is the vertex in $\Q$ associated to $V'$.
 \end{itemize}
If we denote by $\T_{\G}$ the set of all truncated vertices of $\G$ and $\Q_1$ the collection of all the arrows in the induced quiver $Q$, then it is not difficult to see that \[|\Q_1|=\sum_{\alpha\in\G_0\setminus\T_{\G}}\val(\alpha)=\sum_{\alpha\in\G_0}\val(\alpha)-|\T_{\G}|.\] For each nontruncated vertex $\alpha$ in $\G_0$ with $\val(\alpha)=t>1$ and successor sequence $\alpha:V_{i_1}<\cdots<V_{i_t}$, we have a corresponding sequence of arrows in the induced quiver $\Q$

\begin{equation}\label{004}
v_{i_1}\xrightarrow{a^{(\alpha)}_{j_1}} v_{i_2} \xrightarrow{a^{(\alpha)}_{j_2}} \cdots \xrightarrow{a^{(\alpha)}_{j_{t-1}}}v_{i_{t}}\xrightarrow{a^{(\alpha)}_{j_{t}}} v_{i_1}.
\end{equation}
Let $C_l=a^{(\alpha)}_{j_{l}}a^{(\alpha)}_{j_{l+1}}\cdots a^{(\alpha)}_{j_{t}}a^{(\alpha)}_{j_{1}}\cdots a^{(\alpha)}_{j_{l-1}}$ be the oriented cycle in $\Q$, for $1\le l\le t$. We call any of these cycles a \textit{special} $\alpha$-\textit{cycle}. We observe that when $\alpha$ is a nontruncated vertex such that $\val(\alpha)=1$, we have only one special $\alpha$-cycle, which is a loop at the vertex in $\Q$ associated to the unique polygon containing $\alpha$. Now, let $V$ be a fixed polygon in $\G_1$ such that $\alpha$ is a vertex in $V$ and $\occ(\alpha,V)=s\ge1$. Then there are $s$ indices $l_1,\ldots,l_s$ such that $V=V_{i_{l_r}}$, for every $1\le r\le s$. We call any of the cycles $C_{l_1},\ldots,C_{l_s}$ a \textit{special} $\alpha$-\textit{cycle at} $v$, and we denote the collection of these cycles in $\Q$ by $\C_{(\alpha)}^{\,v}$. If we denote by $\CC_{(\alpha)}$ the collection of all the special $\alpha$-cycles in $\Q$ and define $\V_{(\alpha)}=\{\,V\in\G_1\,|\,\alpha\textrm{ occurs in }V\,\}$, the set of polygons where $\alpha$ occurs, then it is easy to prove that
\begin{equation}\label{005}\CC_{(\alpha)}=\bigcup_{V\in\V_{(\alpha)}}\C_{(\alpha)}^{\,v}.\end{equation} Once again, for the particular case of a nontruncated vertex $\alpha$ such that $\val(\alpha)=1$, the collection $\CC_{(\alpha)}$ is just the set consisting of the unique loop at the vertex in $\Q$ associated to the unique polygon where $\alpha$ occurs.

\begin{ejem}\label{007}
The quiver $\Q$ induced by the configuration $\G$ of Example \ref{002} is
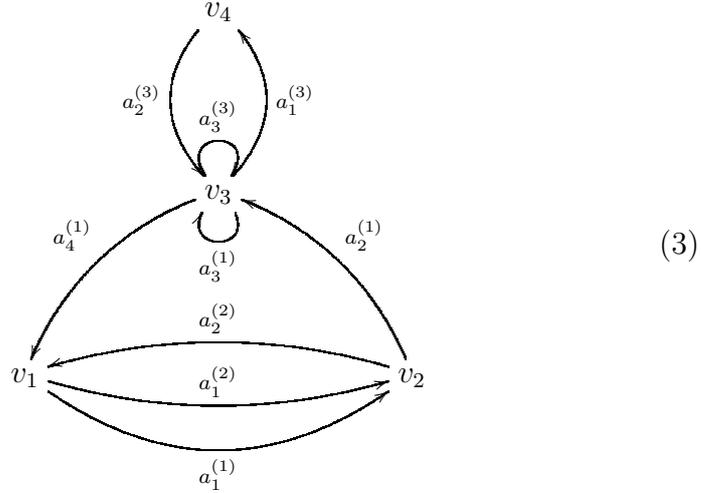
\begin{equation}\label{006}
\begin{split}
\xymatrix{ & & v_4\ar@/_1.5pc/[dd]_{a^{(3)}_2} & & \\  & & & & \\  & &v_3\ar@/_1pc/[lldd]_{a^{(1)}_4}\ar@(ul,ur)^{a^{(3)}_3}\ar@(dr,dl)^{a^{(1)}_3}\ar@/_1.5pc/[uu]_{a^{(3)}_1} & &\\ & & & \\v_1\ar@/_2.4pc/[rrrr]_{a^{(1)}_1}\ar@/_1pc/[rrrr]^{a^{(2)}_1} & & & & v_2\ar@/_1pc/[uull]_{a^{(1)}_2}\ar@/_1pc/[llll]_{a^{(2)}_2}}
\end{split}
\end{equation}
As we can see, every arrow is induced by a successor sequence. For example, the successor sequence at vertex 1 induces in $\Q$ the sequence of arrows \[1:v_1\xrightarrow{a^{(1)}_1} v_2\xrightarrow{a^{(1)}_2}v_3\xrightarrow{a^{(1)}_3}v_3\xrightarrow{a^{(1)}_4}v_1;\] and the successor sequence at 3 induces the sequence of arrows \[3:v_3\xrightarrow{a^{(3)}_1}v_4\xrightarrow{a^{(3)}_2}v_3\xrightarrow{a^{(3)}_3}v_3.\]
\end{ejem}
Let $\G=(\G_0,\G_1,\mu,\oo)$ a Brauer configuration and let $V$ be a polygon of $\G$. We denote by $\overline{V}$ the collection of vertices defined by \begin{equation}\label{a01}\overline{V}:=\{\alpha\in\G_0\,|\,\alpha\textrm{ occurs in }V\}.\end{equation} That is, $\overline{V}$ is the set of the vertices that occur in $V$. The collection $\overline{V}$ is always a set, while $V$ is a multiset.\\
 
 For the induced quiver $\Q$ of the Brauer configuration $\G=(\G_0,\G_1,\mu,\oo)$ define the set
\begin{equation}\label{008}
\CC:=\bigcup_{\alpha\in\G_0\setminus\T_{\G}}\CC_{(\alpha)},
\end{equation}
 and let $f:\CC\to\Q_1$ be the map which sends a special cycle to its first arrow. Now, let $K$ be a field and consider in the path algebra $K\Q$ the following type of relations.
 
\vspace{0.4cm}
 
\noindent\textit{Relations of type one.} It is the subset of $K\Q$ \[\bigcup_{V\in\G_1}\left(\bigcup_{\alpha,\beta\in \overline{V}\setminus\T_{\G}}\left\{\,C^{\mu(\alpha)}-D^{\mu(\beta)}\,|\,C\in\C_{(\alpha)}^{\,v},D\in\C_{(\beta)}^{\,v}\,\right\}\right).\]
\\
\textit{Relations of type two.} It is the subset of $K\Q$ \[\bigcup_{\alpha\in\G_0\setminus\T_{\G}}\left\{\,C^{\mu(\alpha)}f(C)\,|\,C\in\CC_{(\alpha)}\,\right\}.\]
\\
\textit{Relations of type three.} It is the set of all quadratic monomial relations of the form $ab$ in $K\Q$ where $ab$ is not a subpath of any special cycle.\\

\noindent We denote by $\rho_{\G}$ the union of these three types of relations.

\begin{defin}\label{009}
Let $K$ be a field and $\G$ a Brauer configuration. The \textit{Brauer configuration algebra} $\L$ \textit{associated to} $\G$ is defined to be $K\Q/I$, where $\Q$ is the quiver induced by $\G$ and $I$ is the ideal in $K\Q$ generated by the set of relations $\rho_{\G}$.
\end{defin}

\begin{ejem}
Let $\G=(\G_0,\G_1,\mu,\oo)$ be the Brauer configuration given by \[\G_0=\left\{1\right\}, \G_1=\left\{V\right\},\] where $V=\left\{1,1\right\}$, $\mu\equiv1$ and the cyclic order $\oo$ is defined by the only successor sequence \[1:V<V.\] Let $\Q$ be the quiver induced by the sequence of arrows \[1:v\xrightarrow{a}v\xrightarrow{b}v.\] That is, the quiver $\Q$ induced by $\G$ is given by \[\xymatrix{v\ar@(ul,dl)_{b}\ar@(dr,ur)_{a}}.\] Observe that there are only two special $1$-cycles: $ab$ and $ba$. Also observe that the respective relations in $K\Q$ are given by:

\begin{itemize}
\item {\it Relations of type one:} $ab-ba.$
\item {\it Relations of type two:} $aba\textrm{ and }bab.$
\item {\it Relations of type three:} $a^2\textrm{ and }b^2.$
\end{itemize}

If $\L_{\G}$ is the Brauer configuration algebra induced by $\G$, then by the relation of type one we can see that this algebra is commutative. In fact, this Brauer configuration algebra is isomorphic to the algebra \[K[x_1,x_2]/\langle x_1^2,x_2^2\rangle.\]
\end{ejem}

\section{Basis in $v\L w$}\label{sec3}
If $\L$ is the Brauer configuration algebra associated to the Brauer configuration $\G$, $V$ is a polygon of $\G$ and $v$ is the vertex associated to the polygon $V$, a $K$-basis for the space $v\L v$ has been already completely determined in \cite[Proposition 3.3]{mya}. In this section we determine a $K$-basis for the space $v\L w$, where $v$ and $w$ are different vertices associated to the polygons $V$ and $W$, respectively. In order to this, we will apply the same combinatorial technique presented in \cite[Section 3]{mya} with a few adjustments. But first, let's do a brief remind of the notation and terminology.\\

Let $\G$ be a Brauer configuration and let $\Q$ be the induced quiver associated to $\G$. For $\alpha\in\G_0$ a fixed vertex with $\val(\alpha)>1$, let $\alpha:V_{i_1}<\cdots<V_{i_{\val(\alpha)}}$ be its successor sequence. Then in the quiver $\Q$ we have a sequence of arrows

\begin{equation}\label{017}
v_{i_1}\stackrel{a^{(\alpha)}_{j_1}}{\longrightarrow} v_{i_2} \stackrel{a^{(\alpha)}_{j_2}}{\longrightarrow} \cdots \stackrel{a^{(\alpha)}_{j_{\val(\alpha)-1}}}{\longrightarrow}v_{i_{\val(\alpha)}}\stackrel{a^{(\alpha)}_{j_{\val(\alpha)}}}{\longrightarrow} v_{i_1}.
\end{equation}
 Now, if $V$ is a polygon which appears in the successor sequence of $\alpha$ and $\occ(\alpha,V)>1$ then the sequence of arrows in (\ref{017}) can be transformed and represented by the following \textit{Special and Non-Special diagram associated to} $\alpha$.
 
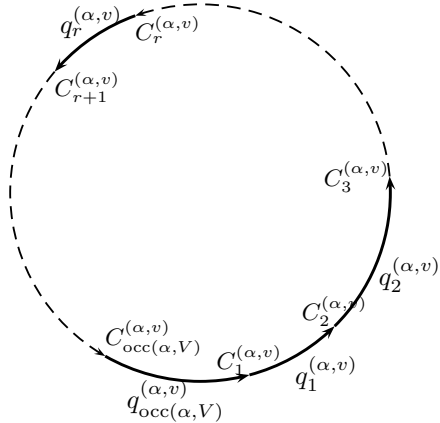
\begin{figure}[H]
\centering
 \begin{pspicture}(-3.2,-3.2)(3.2,3.2)
  \psarc[linewidth=1.1pt]{->}(0,0){2.5}{-120}{-75}
  \psarc[linewidth=1.1pt]{->}(0,0){2.5}{-75}{-45}
  \psarc[linewidth=1.1pt]{->}(0,0){2.5}{-45}{5}
  \psarc[linewidth=0.7pt,linestyle=dashed]{->}(0,0){2.5}{5}{110}
  \psarc[linewidth=1.1pt]{->}(0,0){2.5}{110}{140}
  \psarc[linewidth=0.7pt,linestyle=dashed]{->}(0,0){2.5}{140}{240}
  \rput[t](-0.326,-2.479){\footnotesize $q^{(\alpha,v)}_{\occ(\alpha,V)}$}
  \rput[tl](1.250,-2.165){\footnotesize $q^{(\alpha,v)}_{1}$}
  \rput[tl](2.349,-0.855){\footnotesize $q^{(\alpha,v)}_{2}$}
  \rput[b](-1.434,2.048){\footnotesize $q^{(\alpha,v)}_{r}$}
  \rput[bl](-1.250,-2.165){\scriptsize $C^{(\alpha,v)}_{\occ(\alpha,V)}$}
  \rput[b](0.647,-2.415){\scriptsize $C^{(\alpha,v)}_{1}$}
  \rput[b](1.768,-1.768){\scriptsize $C^{(\alpha,v)}_{2}$}
  \rput[r](2.490,0.218){\scriptsize $C^{(\alpha,v)}_{3}$}
  \rput[tl](-0.855,2.349){\scriptsize $C^{(\alpha,v)}_{r}$}
  \rput[tl](-1.915,1.607){\scriptsize $C^{(\alpha,v)}_{r+1}$}
 \end{pspicture}
 \caption{\small Special and Non-special diagram associated to $\alpha$.}\label{fig3}
\end{figure}

As we can see, the collection $\C^{\,v}_{(\alpha)}=\left\{C^{(\alpha,v)}_{1},\ldots,C^{(\alpha,v)}_{\occ(\alpha,V)}\right\}$ is formed by all the special $\alpha$-cycles at $v$ and the collection $\neg\C^{\,v}_{(\alpha)}=\left\{q^{(\alpha,v)}_{1},\ldots,q^{(\alpha,v)}_{\occ(\alpha,V)}\right\}$ by all the non-special $\alpha$-cycles at $v$.

We divide in \textit{intervals} the diagram in Fig. \ref{fig3} according to the order in which the vertex $v$ occurs on the diagram\footnote{Remember that after to fix one of the occurrences of the vertex $v$ we label it as \textit{1st} $v$ and then in counter clockwise continue labeling the other occurrences of $v$ as \textit{2nd} $v$, \textit{3rd} $v\ldots$ etc \cite[Pages 296 and 297]{mya}.}. This division is made as follows:
\begin{itemize}
\item \textit{1st} $(\alpha,v)$-\textit{interval}: it is the segment formed by all the vertices in $\Q$ that appear, in the same order of occurrence and repetitions included, in the special and non-special diagram associated to $\alpha$ between the \textit{1st} $v$ and the \textit{2nd} $v$.

\item \textit{2nd} $(\alpha,v)$-\textit{interval}: it is the segment formed by all the vertices in $\Q$ that appear, in the same order of occurrence and repetitions included, in the special and non-special diagram associated to $\alpha$ between the \textit{2nd} $v$ and the \textit{3rd} $v$.
\item And so on$\ldots$
\end{itemize}

When $\alpha$ and $v$ are clear from the context we just call them intervals instead of $(\alpha,v)$-intervals. Now, in Fig. \ref{fig3} we delete all the $q^{(\alpha,v)}_{i}$\,'s and instead of the $C^{(\alpha,v)}_i$\,'s we put the corresponding $i$-\textit{th} $v$'s, and setting $s=\occ(\alpha,V)$, the diagram that represents all the intervals it would look like  Fig. \ref{fig4}.

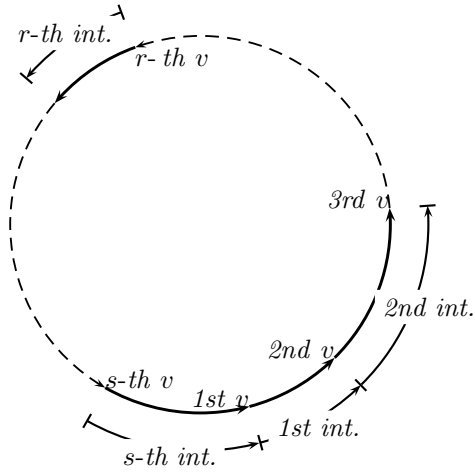
\begin{figure}[H]
\centering
 \begin{pspicture}(-3.2,-3.2)(3.2,3.2)
  \psarc[linewidth=1.1pt]{->}(0,0){2.5}{-120}{-75}
  \psarc[linewidth=1.1pt]{->}(0,0){2.5}{-75}{-45}
  \psarc[linewidth=1.1pt]{->}(0,0){2.5}{-45}{5}
  \psarc[linewidth=0.7pt,linestyle=dashed]{->}(0,0){2.5}{5}{110}
  \psarc[linewidth=1.1pt]{->}(0,0){2.5}{110}{140}
  \psarc[linewidth=0.7pt,linestyle=dashed]{->}(0,0){2.5}{140}{240}
  
  \psarc[linewidth=0.8pt]{|->|}(0,0){3}{-75}{-45}
  \psarc[linewidth=0.8pt]{->|}(0,0){3}{-45}{5}
  \psarc[linewidth=0.8pt]{|->|}(0,0){3}{110}{140}
  \psarc[linewidth=0.8pt]{|->}(0,0){3}{-120}{-75}
  
  \rput*(1.55,-2.685){\footnotesize\textit{1st int.}}
  \rput*(3.007,-1.094){\footnotesize\textit{2nd int.}}
  \rput*(-1.778,2.539){\footnotesize $r$-\textit{th int.}}
  \rput*(-0.405,-3.073){\footnotesize $s$-\textit{th int.}}
  
  \rput[br](0.647,-2.415){\footnotesize\textit{1st} $v$}
  \rput[br](1.768,-1.768){\footnotesize\textit{2nd} $v$}
  \rput[br](2.490,0.218){\footnotesize\textit{3rd} $v$}
  \rput[tl](-0.855,2.349){\footnotesize $r$-\,\textit{th} $v$}
  \rput[bl](-1.250,-2.165){\footnotesize $s$-\textit{th} $v$}
 \end{pspicture}
 \caption{\small The $(\alpha,v)$-intervals}\label{fig4}
\end{figure}

Now, if $W\in\V_{(\alpha)}$, where $\V_{(\alpha)}=\left\{W\in\G_1\,|\,\alpha\textrm{ occurs in }W\right\}$, with $W\neq V$, and $w$ is the vertex in $\Q$ associated to $W$, we denote by $\occ^{(\alpha,v)}_i(w)$ the number of times that the vertex $w$ occurs in the $i$-th interval, for every $1\le i\le\occ(\alpha,V)$. It is clear that $\occ^{(\alpha,v)}_i(w)\ge0$, for all $1\le i\le\occ(\alpha,V)$, and \begin{equation}\label{024}\sum_{i=1}^{\occ(\alpha,V)}\occ^{(\alpha,v)}_i(w)=\occ(\alpha,W).\end{equation}

\begin{lem}\label{015}
Let $\L=K\Q/I$ be the Brauer configuration algebra associated to the Brauer configuration $\G$. Let $V,W$ be polygons in $\G_1$, and let $v,w$ be the respective associated vertices in $\Q$. Then \[v\L w\neq0\iff\overline{V}\cap\overline{W}\neq\emptyset.\]
\end{lem}
\begin{proof}
If $v\L w\neq0$ then there exists a path from $v$ to $w$ contained in some $\alpha$-cycle. Hence $\alpha\in\overline{V}\cap\overline{W}$. Conversely, if $\alpha\in\overline{V}\cap\overline{W}$ then both vertices $v$ and $w$ occur in any $\alpha$-cycle, therefore we have at least a path from $v$ to $w$ such that its class in $\L$ is not zero, hence $v\L w\neq0$.
\end{proof}
\begin{propo}\label{018}
Let $\L=K\Q/I$ be the Brauer configuration algebra induced by $\G=(\G_0,\G_1,\mu,\oo)$. If $V,W\in\G_1$, with $V\neq W$, then \[\textrm{dim}_Kv\L w=\sum_{\alpha\in\overline{V}\cap\overline{W}}\mu(\alpha)\occ(\alpha,V)\occ(\alpha,W),\] where $v$ and $w$ are the vertices in $\Q$ associated to $V$ and $W$ respectively.
\end{propo}
\begin{proof}
Let $V,W\in\G_1$ be two different polygons of the configuration, and let $v$ and $w$ the vertices in $\Q$ associated to $V$ and $W$, respectively. By Lemma \ref{015} we have that dim$_Kv\L w=0\iff\overline{V}\cap\overline{W}=\emptyset$. Thus, let $\alpha$ be a vertex in $\overline{V}\cap\overline{W}$ and first let's suppose that $\occ(\alpha,V)=1$ and $\occ(\alpha,W)\ge1$. In this case, we have an only special $\alpha$-cycle at $v$, given by $C^{(\alpha,v)}_1$. Let $w$ be one fixed vertex of those that appear in the special and non-special diagram associated to $\alpha$. If we denote by $q_{(v,w)}$ the composition of all the arrows in the special and non-special diagram between $v$ and $w$, we obtain $\mu(\alpha)$ paths in $\Q$ from $v$ to $w$ of the form $\left(C^{(\alpha,v)}_1\right)^kq_{(v,w)}$, for each $0\le k<\mu(\alpha)$. Now, if we do the same for the rest of $w$'s in the special and non-special diagram associated to $\alpha$ we obtain a total of $\mu(\alpha)\occ(\alpha,W)=\mu(\alpha)\occ(\alpha,V)\occ(\alpha,W)$ paths from $v$ to $w$ in $\Q$ associated to the vertex $\alpha$.

Now, suppose that $\occ(\alpha,V)>1$, $\occ(\alpha,W)\ge1$, and set $s=\occ(\alpha,V)$. Without loss of generality, also suppose that  $\occ^{(\alpha,v)}_s(w)>0$, i.e, there is at least one $w$ in the $s$-th interval in the special and non-special diagram associated to $\alpha$. Let $w$ be one of the occurrences of this vertex in the $s$-th interval, and let $q_{(v,w)}$ be the composition of all the arrows in the $s$-th interval between the $s$-\textit{th} $v$ and $w$. This is represented in Fig. \ref{fig2} 
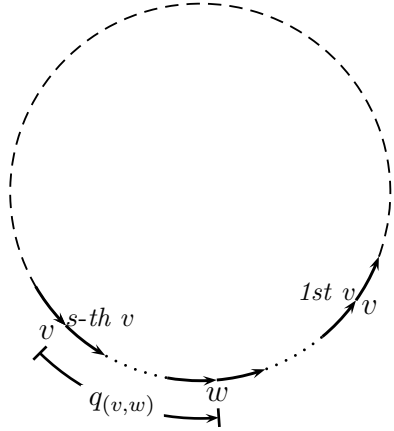
\begin{figure}[H]
\centering
 \begin{pspicture}(-3.2,-3.2)(3.2,3.2)
  \psarc[linewidth=1.1pt]{->}(0,0){2.5}{-150}{-135}
  \psarc[linewidth=1.1pt]{->}(0,0){2.5}{-135}{-120}
  \psarc[linewidth=1.1pt,linestyle=dotted]{-}(0,0){2.5}{-120}{-100}
  \psarc[linewidth=1.1pt]{->}(0,0){2.5}{-100}{-85}
  \psarc[linewidth=1.1pt]{->}(0,0){2.5}{-85}{-70}
  \psarc[linewidth=1.1pt,linestyle=dotted]{-}(0,0){2.5}{-70}{-50}
  \psarc[linewidth=1.1pt]{->}(0,0){2.5}{-50}{-35}
  \psarc[linewidth=1.1pt]{->}(0,0){2.5}{-35}{-20}
  \psarc[linewidth=0.7pt,linestyle=dashed]{->}(0,0){2.5}{-20}{225}
  \rput[bl](-1.768,-1.768){\footnotesize $s$-\textit{th} $v$}
  \rput[tr](-1.768,-1.768){$\,\,v\,\,$}
  \rput[t](0.227,-2.590){$\,\,w\,\,$}
  \rput[br](2.048,-1.434){\footnotesize\textit{1st} $v$}
  \rput[tl](2.048,-1.434){$\,v\,$}
  \psarc[linewidth=1.01pt]{|->|}(0,0){3}{-135}{-85}
  \rput*(-1.026,-2.819){\small $q_{(v,w)}$}
  \end{pspicture}
  \caption{\small Here $q_{(v,w)}$ represents the composition of all the arrows in the special and non-special diagram associated to $\alpha$, between the $s$-\textit{th} $v$ and the fixed vertex $w$ which is in the $s$-th interval.}\label{fig2}
 \end{figure}
 Then all the possible paths in $\Q$ associated to $\alpha$ that start at $v$, any $v$ of the special and non-special diagram, and finish at this vertex $w$ can be listed as
 \[\begin{array}{l}
 \left(C^{(\alpha,v)}_{s}\right)^kq_{(v,w)}\\
 \left(C^{(\alpha,v)}_{s-1}\right)^kq^{(\alpha,v)}_{s-1}q_{(v,w)}
 \end{array}
 \]
 
 \[\vdots\]
 
 \[\begin{array}{l}
 \left(C^{(\alpha,v)}_{2}\right)^kq^{(\alpha,v)}_{2}\cdots q^{(\alpha,v)}_{s-1}q_{(v,w)}\\
 \left(C^{(\alpha,v)}_{1}\right)^kq^{(\alpha,v)}_{1}\cdots q^{(\alpha,v)}_{s-1}q_{(v,w)}
 \end{array}\]
 for each $0\le k<\mu(\alpha)$. So, we obtain a total of $\mu(\alpha)\occ(\alpha,V)$ elements from all these lists. Now, if we do the same for the rest of the $w$'s that are in the $s$-th interval we will obtain $\mu(\alpha)\occ(\alpha,V)\occ^{(\alpha,v)}_s(w)$ paths in $\Q$ associated to $\alpha$ that start at $v$ and finish at $w$, where $w$ is a vertex in the $s$-th interval. Finally, doing this very same reasoning for the rest of the $(\alpha,v)$-intervals in the special and non-special diagram associated to $\alpha$, we can conclude that the total number of paths in $\Q$ associated to the vertex $\alpha$ that start at $v$ and finish at $w$ is equal to
 \begin{eqnarray*}
 \sum_{i=1}^{s}\mu(\alpha)\occ(\alpha,V)\occ^{(\alpha,v)}_i(w) & = & \mu(\alpha)\occ(\alpha,V)\sum_{i=1}^{s}\occ^{(\alpha,v)}_i(w)\\
  & = & \mu(\alpha)\occ(\alpha,V)\occ(\alpha,W).
 \end{eqnarray*}
 Therefore it follows that \[\textrm{dim}_Kv\L w=\sum_{\alpha\in\overline{V}\cap\overline{W}}\mu(\alpha)\occ(\alpha,V)\occ(\alpha,W).\]
\end{proof}

\section{The Cartan matrix of a Brauer configuration algebra}\label{sec5}

We start this section by presenting a brief description of the Cartan matrix of a basic finite dimensional associative $K$-algebra $\L$ with respect to a complete set of primitive orthogonal idempotents. We also present a short proof about the fact that if $\L$ is a symmetric algebra then the associated Cartan matrix is symmetric. We finalize the section by expressing the Cartan matrix of $\L$  when it is a Brauer configuration algebra (see Proposition \ref{046}).

In this section we assume that $K$ is an algebraically closed field.

\begin{defin}\label{010}
Let $\L$ be a basic finite dimensional associative $K$-algebra with a complete set $\{e_1,\ldots,e_n\}$ of primitive orthogonal idempotents. Let $\CM_{\L}=\left(c_{i,j}\right)$ be the $n\times n$ matrix with entries in $\Z_{\ge0}$ where \[c_{ij}=\dim_Ke_j\L e_i,\textrm{ for }1\le i,j\le n.\] This matrix is called the \textit{Cartan matrix} of $\L$.
\end{defin}

If $\{e'_1,\ldots,e'_m\}$ is another complete set of primitive orthogonal idempotents of $\L$, then by the Krull-Schmidt theorem it follows that $m=n$, and if $\mathcal{C}'=\left(c'_{ij}\right)$ is a $n\times n$ matrix given by \[c'_{ij}=\dim_Ke'_j\L e'_i,\textrm{ for }1\le i,j\le n,\] then $\mathcal{C}'$ is obtained from $\CM_{\L}$ by a permutation of its rows and columns. Thus, when we refer to the Cartan matrix of a $K$-algebra $\L$ it means the Cartan matrix defined with respect to a given complete set $\{e_1,\ldots,e_n\}$ of primitive orthogonal idempotents of $\L$, and where $\L$ is basic and finite dimensional.\\

If $M$ is a finitely generated (right) $\L$-module its \textit{dimension vector} is defined as the vector given by \[\textrm{\bf dim}\,M=\left(\textrm{dim}_KMe_1,\ldots,\textrm{dim}_KMe_n\right)^t.\](See \cite[Definition III.3.1]{rep1}) .

If $P_1,\ldots,P_n$ is the collection of all indecomposable projective modules of $\L$, we can assume that $P_i=e_i\L$ for every $1\le i\le n$. Then, we obtain that the Cartan matrix of $\L$ can be expressed as\[\CM_{\L}=\left(\,\textrm{\bf dim}\,P_1\,\cdots\,\textrm{\bf dim}\,P_n\,\right).\] With this list of indecomposable projective modules  it follows that the collection $S_1,\ldots,S_n$ of simple $\L$-modules would be given by \[S_i=e_i\L/\rad e_i\L,\textrm{ for all }1\le i\le n.\]

Now, if $I_1,\ldots,I_n$ is the list of all indecomposable injective modules of $\L$, we can also assume that $I_i=\hom_K(\L e_i,K)$, for each $1\le i\le n$.
 
if $\L$ is a \textit{symmetric} algebra, then there exists an isomorphism of modules \[P_i\cong I_i,\textrm{ for each }1\le i\le n.\]

We have the following proposition.
\begin{propo}\label{020}
Let $\L$ be a basic finite dimensional $K$-algebra. If $\L$ is a symmetric algebra, then its Cartan matrix $\CM_{\L}$ is symmetric.
\end{propo}
\begin{proof}
Let $\CM_{\L}$ denote the Cartan matrix of $\L$. By \cite[Proposition III.3.8]{rep1} we have the equalities
\begin{eqnarray*}
\textrm{\bf dim}\,P_i & = & \CM_{\L}\cdot\textrm{\bf dim}\,S_i,\\\textrm{\bf dim}\,I_i & = & \CM_{\L}^t\cdot\textrm{\bf dim}\,S_i.
\end{eqnarray*}
Because $\L$ is symmetric we obtain that $\textrm{\bf dim}\,P_i=\textrm{\bf dim}\,I_i$, and hence \[\CM_{\L}\cdot\textrm{\bf dim}\,S_i=\CM_{\L}^t\cdot\textrm{\bf dim}\,S_i.\] This equality is satisfied for all $1\le i\le n$. Now, it is a fact that the vectors $\textrm{\bf dim}\,S_1,\ldots,\textrm{\bf dim}\,S_n$ form the standard basis of the $\Z$-module $\Z^n$, therefore it follows that \[\CM_{\L}=\CM_{\L}^t.\]
\end{proof}

We now present the explicit expression of the Cartan matrix of a Brauer configuration algebra. 

\begin{propo}\label{046}
Let $\L=K\Q/I$ be the Brauer configuration algebra associated to $\G=(\G_0,\G_1,\mu,\oo)$. If $\CM_{\L}=\left(c_{v,w}\right)_{V,W\in\G_1}$ is the Cartan matrix of $\L$, then

\begin{equation*}\label{021}
c_{v,w}=\left\{\begin{array}{cr}2+\sum\limits_{\alpha\in\overline{V}}\occ(\alpha,V)\left(\occ(\alpha,V)\mu(\alpha)-1\right), & V=W;\\\sum\limits_{\alpha\in\overline{V}\cap\overline{W}}\mu(\alpha)\occ(\alpha,V)\occ(\alpha,W), & V\neq W.\end{array}\right.
\end{equation*}
\end{propo}

\begin{proof}
It follows from Proposition \ref{018} and \cite[Proposition 3.3]{mya}.
\end{proof}
From the expression of the entries $c_{v,w}$ it is clear that the matrix $\CM_{\L}$ is symmetric,  and it must be like this because Brauer configuration algebras are symmetric.
\begin{ejem}\label{022}
Let $\G=(\G_0,\G_1,\mu,\oo)$ be a Brauer configuration such that $\mu\equiv1$ and the polygons in $\G_1=\left\{V_1,\ldots,V_n\right\}$ are sets. We can affirm that for any $V\in\G_1$ \[\occ(\alpha,V)\le1,\textrm{ for all }\alpha\in\G_0.\] Let $\L$ be the Brauer configuration algebra induced by $\G$ and let $\CM_{\L}$ be its Cartan matrix. Then by Proposition \ref{046} we have that
\[\CM_{\L}=\left(\begin{array}{cccc}2 & |V_1\cap V_2| & \cdots & |V_1\cap V_n|\\|V_1\cap V_2| & 2 &  & |V_2\cap V_n|\\ \vdots & \vdots & \ddots & \vdots\\|V_1\cap V_n| & |V_2\cap V_n| & \cdots & 2\end{array}\right).\]
\end{ejem}

\begin{ejem}\label{023}
Let $\G$ be the Brauer configuration from Example \ref{002}. It is not difficult to see that 
\begin{multicols}{3}
\begin{enumerate}[]
\item $\overline{V}_1\cap\overline{V}_2=\{1,2\}$,
\item $\overline{V}_1\cap\overline{V}_3=\{1\}$,
\item $\overline{V}_1\cap\overline{V}_4=\emptyset$,
\item $\overline{V}_2\cap\overline{V}_3=\{1\}$,
\item $\overline{V}_2\cap\overline{V}_4=\emptyset$,
\item $\overline{V}_3\cap\overline{V}_4=\{3\}$.
\end{enumerate}
\end{multicols}
Let $\L$ be the Brauer configuration algebra associated to $\G$ and let $\CM_{\L}=\left(c_{i,j}\right)_{1\le i,j\le 4}$ be the respective Cartan matrix. Thus, the entries of the main diagonal are equal to
\[
\begin{array}{rclcr}
c_{1,1} & = & 2+1\cdot(1\cdot2-1)+1\cdot(1\cdot2-1) & = &4,\\
 c_{2,2} & = & 2+1\cdot(1\cdot2-1)+1\cdot(1\cdot2-1) & = &4,\\
 c_{3,3} & = & 2+2\cdot(2\cdot2-1)+2\cdot(1\cdot2-1) & = & 10,\\
  c_{4,4} & = & 2, & & 
\end{array}
\]
and the other entries are given by
\[
\begin{array}{rclcl}
c_{1,2} & = & 2\cdot1\cdot1+2\cdot1\cdot1 & = & 4,\\
c_{1,3} & = & 2\cdot1\cdot2 & = & 4,\\
c_{1,4} & = & 0, & &\\
c_{2,3} & = & 2\cdot1\cdot2 & = & 4,\\
c_{2,4} & = & 0, & &\\
c_{3,4} & = & 1\cdot2\cdot1 & = & 2,
\end{array}
\]
then the Cartan matrix is equal to
\[\CM_{\L}=\left(\begin{array}{cccc}4 & 4 & 4 & 0\\4 & 4 & 4 & 0\\4 & 4 & 10 & 2\\0 & 0 & 2 & 2\end{array}\right).\]
As expected, the sum of all the entries of $\CM_{\L}$ must be equal to the vector dimension of $\L$, which is given\footnote{By \cite[Proposition 3.13]{brau} the expression to compute the vector dimension of $\L$ is  $2|\G_1|+\sum\limits_{\alpha\in\G_0}\val(\alpha)(\mu(\alpha)\val(\alpha)-1)$.} by

{\small
\begin{eqnarray*}
\textrm{dim}_K\L & = & 2\cdot4+4\cdot(2\cdot4-1)+2\cdot(2\cdot2-1)+3\cdot(1\cdot3-1),\\
 & = & 8+28+6+6,\\
 & = & 48.
\end{eqnarray*}
}

\end{ejem}

\section{The length of an indecomposable projective module associated to a Brauer configuration algebra}\label{sec6}

In this section, the formula of the length of the $\L$-module resulting from the sum of a class of uniserial modules contained in an indecomposable projective $\L$-module, where $\L$ is a Brauer configuration algebra, is determined (Theorem \ref{032}). Then, this formula is used to determine the length of any indecomposable projective $\L$-module (Corollary \ref{044}), and other formulas too (Propositions \ref{050} and \ref{051}).\\ 

Given a (right) $\L$-module $M$, with $\L$ a finite-dimensional $K$-algebra, the length of $M$ is defined as the number of non-zero submodules of $M$ in a composition series, in case that such a composition series exists. This numerical value for $M$ is usually denoted by $\ell(M)$. We start by introducing the nonprojective uniserial modules associated to a Brauer configuration algebra as also the notation introduced by the authors in \cite[Section 3.4]{brau}.

\vspace{0.4cm}

Let $\L=K\Q/I$ be a Brauer configuration algebra associated to a Brauer configuration $\G$, and let $V\in\G_1$ be a polygon with $v$ its associated vertex in $\Q$. If $P_V$ is the indecomposable projective $\L$-module associated to $v$ and $S_V$ the simple $\L$-module, we know that these can be identified respectively by $P_V=v\L$ and $S_V=v\L/\rad\,v\L$. For a nontruncated vertex $\alpha\in\G_0$ assume that $\alpha\in\G_0\cap\overline{V}$ and let $C\in\C_{(\alpha)}^{\,v}$ be a special $\alpha$-cycle with $C=a^{(\alpha)}_{j_{1}}a^{(\alpha)}_{j_{2}}\cdots a^{(\alpha)}_{j_{\val(\alpha)}}$ and $\val(\alpha)>1$. Let's represent the arrows forming the special cycle $C$ by the following figure.

\begin{figure}[H]
\centering
 \begin{pspicture}(-3.2,-3.2)(3.2,3.2)
 \psarc[linewidth=1.1pt]{->}(0,0){2}{-125}{-85}
 \psarc[linewidth=1.1pt]{->}(0,0){2}{-85}{-45}
  \psarc[linewidth=1.1pt]{->}(0,0){2}{-45}{-5}
  \psarc[linewidth=0.7pt,linestyle=dashed]{->}(0,0){2}{-5}{75}
  \psarc[linewidth=1.1pt]{->}(0,0){2}{75}{115}
  \psarc[linewidth=1.1pt]{->}(0,0){2}{115}{155}
  \psarc[linewidth=0.7pt,linestyle=dashed]{->}(0,0){2}{155}{235}
    \rput[t](0.183,-2.092){$C$}
    \rput[t](0.930,-1.994){\footnotesize{$a_{j_{1}}^{(\alpha)}$}}
    \rput[t](1.994,-0.930){\footnotesize{$a_{j_{2}}^{(\alpha)}$}}
    \rput[b](-0.192,2.192){\footnotesize{$a_{j_{k}}^{(\alpha)}$}}
    \rput[b](-1.556,1.556){\footnotesize{$a_{j_{k+1}}^{(\alpha)}$}}
    \rput[t](-0.569,-2.125){\footnotesize{$a_{j_{\val(\alpha)}}^{(\alpha)}$}}
    \rput[b](0.165,-1.893){$v_{j_1}$}
    \rput[br](1.343,-1.343){$v_{j_2}$}
    \rput[r](1.893,-0.165){$v_{j_3}$}
    \rput[t](0.492,1.835){$v_{j_k}$}
    \rput[tl](-0.803,1.722){$v_{j_{k+1}}$}
    \rput[l](-1.722,0.803){$v_{j_{k+2}}$}
    \rput[l](-1.117,-1.537){$v_{j_{\val(\alpha)}}$}
 \end{pspicture}
 \caption{Graphic representation of the special $\alpha$-cycle $C$.}
 \end{figure}
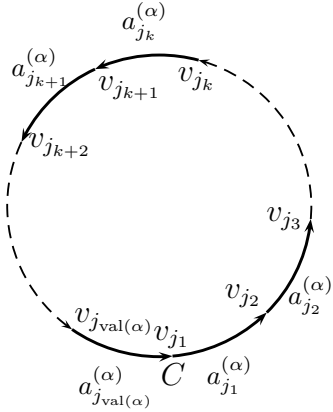
 
 As we can see from this graphic representation of the special $\alpha$-cycle $C$, the arrow $a^{(\alpha)}_{j_t}$ starts at $v_{j_t}$ and ends at $v_{j_{t+1}}$ for all $1\le t\le\val(\alpha)$, as also $v_{j_{\val(\alpha)}+1}=v_{j_1}=v$. Now, for $1\le j\le\mu(\alpha)\val(\alpha)$ such that $j=\val(\alpha)k+l$, with $0\le k<\mu(\alpha)$ and $0\le l<\val(\alpha)$, let $U_j=U_j(C)$ be the submodule of $v\L$ given by:
 \begin{itemize}
     \item if $l\neq0$ then $U_j$ has $K$-basis \[\left\{\overline{C^ka^{(\alpha)}_{j_1}\cdots a^{(\alpha)}_{j_l}},\overline{C^ka^{(\alpha)}_{j_1}\cdots a^{(\alpha)}_{j_{l+1}}},\ldots,\overline{C^{k+1}},\overline{C^{k+1}a^{(\alpha)}_{j_1}},\ldots,\overline{C^{\mu(\alpha)}}\right\};\]
     \item if $l=0$ then $U_j$ has $K$-basis \[\left\{\overline{C^k},\overline{C^ka^{(\alpha)}_{j_1}},\overline{C^ka^{(\alpha)}_{j_1}a^{(\alpha)}_{j_2}},\ldots,\overline{C^{\mu(\alpha)}}\right\}.\]
 \end{itemize}
 
 According\footnote{By \cite[Lemma 3.7]{brau} these are the only non-projective uniserial $\L$-modules.} to \cite[Section 3.4]{brau} $U_j$ is an uniserial $\L$-module containing $U_{j+1}$ and such that $U_{j}/U_{j+1}$ is isomorphic to the simple module associated to $v_{j_{l+1}}$, for each $1\le j\le\mu(\alpha)\val(\alpha)$. Thus, we have the following composition series of the uniserial $\L$-module $U_1$ \begin{equation}\label{025}0\subset U_{\mu(\alpha)\val(\alpha)}\subset U_{\mu(\alpha)\val(\alpha)-1}\subset\cdots\subset U_2\subset U_1.\end{equation}
\begin{obs}\label{028}
If the special $\alpha$-cycle $C$ is a loop then necessarily $\val(\alpha)=1$ and $\mu(\alpha)>1$, hence the $K$-basis of the module $U_j$ is given by \[\left\{\overline{C}^{j},\overline{C}^{j+1},\ldots,\overline{C}^{\mu(\alpha)}\right\},\] for each $1\le j\le\mu(\alpha)$.
\end{obs}

 The following proposition has a straightforward proof. 
 \begin{propo}\label{027}
 Let $C$ be an special $\alpha$-cycle. Then \[\textrm{dim}_KU_j(C)=\mu(\alpha)\val(\alpha)-j+1,\] for all $1\le j\le\mu(\alpha)\val(\alpha)$.
 \end{propo}
 
 Only in this section, we use the same symbols of the multiplicity function and the valence to define the following functions. This is obviously an abuse of notation but it will seem natural in the computations that we are going to develop. So, let $\mu,\val:\CC\to\mathbb{Z}_{>0}$ be functions defined respectively by:
 
 For each $\alpha\in\G_0\setminus\T_{\G}$
 \begin{eqnarray*}
 \mu(C) & = & \mu(\alpha),\\ \val(C) & = & \val(\alpha),
 \end{eqnarray*}
 for all $C\in\CC_{(\alpha)}$ (see (\ref{005}) and (\ref{008}) to remind definitions).

As mentioned in \cite{brau2} an alternative way of studying Brauer configuration algebras is by using the concept of a \textit{defining pair} \cite[Section 3]{brau2}. According to the notation that we are implementing the pair given by $(\CC,\mu)$ is a defining pair of $\Q$. Thus we may say that the Brauer configuration algebra $\L=K\Q/I$ induced by $\G$ coincides with the algebra defined by $(\CC,\mu)$. The Proposition \ref{027} can be restated as follows.

\begin{propo}\label{036}
Let $C$ be a special cycle. Then \[\dime_K{U_j(C)}=\mu(C)\val(C)-j+1,\] for all $1\le j\le \mu(C)\val(C)$.
\end{propo}

For $C$ and $C'$ different special cycles at the vertex $v$ and $1\le j\le\mu(C')\val(C')$, the respective $K$-basis of $U_1(C)$ and $U_j(C')$ have a unique element in common which is given by \[\overline{C^{\mu(C)}}=\overline{C'^{\mu(C')}}.\] Now, if $\Delta$ is a nonempty collection of special cycles at $v$ such that $C'\notin\Delta$ it is straightforward to see that also in this case the $K$-basis of $\sum_{C\in\Delta}U_1(C)$ and $U_j(C')$ have a unique element in common, namely $\overline{C'^{\mu(C')}}$. Hence, we can affirm that
\begin{equation*}
\left(\sum_{C\in\Delta}U_1(C)\right)\cap U_j(C')
\end{equation*}
has $K$-basis $\left\{\overline{C'^{\mu(C')}}\right\}$ and therefore
\begin{equation}\label{031}
    \dime_K\left[\left(\sum_{C\in\Delta}U_1(C)\right)\cap U_j(C')\right]=1.
\end{equation}
\begin{propo}\label{030}
Let $V$ be a polygon in $\G$ and let $\Delta$ be a nonempty collection of special cycles at the vertex $v$. If $C'$ is a special cycle at $v$ such that $C'\notin\Delta$, then the quotient \[\left(\sum_{C\in\Delta}U_1(C)+U_{j}(C')\right)\left/\left(\sum_{C\in\Delta}U_1(C)+U_{j+1}(C')\right)\right.\] is a simple module for each $1\le j\le\mu(C')\val(C')-1$.
\end{propo}
\begin{proof}
Let $N=\sum_{C\in\Delta}U_1(C)$, $U'_j=U_j(C')$ and $U'_{j+1}=U_{j+1}(C')$. By Proposition \ref{036} it is straightforward that  $\dim_K U'_j/U'_{j+1}=1$, then by (\ref{031}) we obtain
\begin{eqnarray*}
\dime_K(N+U'_{j})/(N+U'_{j+1}) & = & \dime_K(N+U'_{j})-\dime_K(N+U'_{j+1})\\ & = & 1+\dime_KN\cap U'_{j+1}-\dime_KN\cap U'_{j}\\ & = & 1
\end{eqnarray*}
Thus, necessarily the quotient $(N+U'_j)/(N+U'_{j+1})$ is a simple $\L$-module.
\end{proof}

Consider the particular case $j=\mu(C')\val(C')-1$ in the previous proposition. We see that 

{\small\begin{multline*}
\left(\sum_{C\in\Delta}U_1(C)+U'_{\mu(C')\val(C')-1}\right)\left/\left(\sum_{C\in\Delta}U_1(C)+U'_{\mu(C')\val(C')}\right)\right.=\\
\left.\left(\sum_{C\in\Delta}U_1(C)+U'_{\mu(C')\val(C')-1}\right)\right/\sum_{C\in\Delta}U_1(C)
\end{multline*}}
Now, if $\mu(C')\val(C')>2$ then the series of submodules 
{\small\begin{multline}\label{037}
\sum_{C\in\Delta}U_1(C)\subset\sum_{C\in\Delta}U_1(C)+U'_{\mu(C')\val(C')-1}\subset\cdots\subset\sum_{C\in\Delta}U_1(C)+U'_1\\
=\sum_{C\in\Delta\cup\left\{C'\right\}}U_1(C)
\end{multline}}
is a segment of a composition series. And if $\mu(C')\val(C')=2$ then the inclusion
{\small\begin{equation}\label{038}\sum_{C\in\Delta}U_1(C)\subset\sum_{C\in\Delta}U_1(C)+U'_1=\sum_{C\in\Delta\cup\left\{C'\right\}}U_1(C)\end{equation}} is a trivial segment of a composition series.

\begin{theo}\label{032}
Let $V$ be a polygon in $\G$ and let $\Delta$ be a nonempty collection of special cycles at the vertex $v$. Then \[\ell\left(\sum_{C\in\Delta}U_1(C)\right)=\sum_{C\in\Delta}\mu(C)\val(C)-|\Delta|+1.\]
\end{theo}
\begin{proof}
By Proposition \ref{036} the equality is obvious when $\Delta$ has an only element. So, let's suppose that $|\Delta|>1$ and it is given by $\Delta=\left\{C_1,\ldots,C_r\right\}$. For each $1\le k\le r$ let $\Delta_k$ be defined as \[\Delta_k=\cup_{j=1}^k\left\{C_j\right\}.\] Setting $U_j=U_j(C_1)$ we start with the composition series of $U_1$,

\begin{equation}\label{039}
0\subset U_{\mu(C_1)\val(C_1)}\subset U_{\mu(C_1)\val(C_1)-1}\subset\cdots\subset U_2\subset U_1.
\end{equation}
By Proposition \ref{030} and using both segments of composition series of the type in (\ref{037}) and of the type in (\ref{038}), we construct successively a composition series segment of the form
{\footnotesize
\begin{equation}\label{040}
U_1=\sum_{C\in\Delta_1}U_1(C)\subset\cdots\subset\sum_{C\in\Delta_2}U_1(C)\subset\cdots\subset\sum_{C\in\Delta_r}U_1(C)=\sum_{C\in\Delta}U_1(C).
\end{equation}}
For each $1\le k\le r-1$, we can see in (\ref{040}) that the number of consecutive submodules from $\sum_{C\in\Delta_k}U_1(C)$ to $\sum_{C\in\Delta_{k+1}}U_1(C)$ is equal to $\mu(C_{k+1})\val(C_{k+1})-1$, while $\ell(U_1)=\mu(C_1)\val(C_1)$ in (\ref{039}). Gluing the segments in (\ref{039}) and (\ref{040}) we finally obtain an explicit composition series of the module $\sum_{C\in\Delta}U_1(C)$, therefore its length is equal to
\begin{eqnarray*}
\ell\left(\sum_{C\in\Delta}U_1(C)\right) & = & \sum_{C\in\Delta}\left(\mu(C)\val(C)-1\right)+1,\\ & = & \sum_{C\in\Delta}\mu(C)\val(C)-|\Delta|+1.
\end{eqnarray*}
\end{proof}

\begin{cor}\label{044}
Let $\L=K\Q/I$ be the Brauer configuration algebra induced by $\G$, and let $V$ be a polygon of $\G$ with $v$ its associated vertex in $\Q$. Then \[\ell\left(v\L\right)=2+\sum_{\alpha\in\overline{V}}\occ(\alpha,V)(\val(\alpha)\mu(\alpha)-1).\]
\end{cor}

\begin{proof}
It is not difficult to see that the collection of all the special cycles at $v$ is given by \[\Delta=\bigcup_{\alpha\in\overline{V}}\C_{(\alpha)}^{\,v}.\] On the other hand, in the proof of \cite[Theorem 3.10]{brau} is mentioned that \[\rad(v\L)=\sum_{C\in\Delta}U_1(C),\]

then by the Theorem \ref{032}
\begin{equation}\label{043}
\ell(\rad(v\L))=\sum_{C\in\Delta}\mu(C)\val(C)-|\Delta|+1.
\end{equation}
From definitions of Section \ref{sec02} any two different elements in the collection $\{\C_{(\alpha)}^{\,v}\,|\,\alpha\in\overline{V}\}$ are disjoint and $|\C_{(\alpha)}^{\,v}|=\occ(\alpha,V)$, for all $\alpha\in\overline{V}$, then
\begin{equation}\label{041}
|\Delta|=\sum_{\alpha\in\overline{V}}\occ(\alpha,V).
\end{equation}
Now, according to the definition of the functions $\mu,\val:\CC\to\mathbb{Z}_{>0}$ we have, in this case, that
\[\begin{array}{cl}
\begin{array}{rcc}\mu(C)&=&\mu(\alpha)\\\val(C)&=&\val(\alpha)\end{array}; & \forall\alpha\in\overline{V},\,\forall C\in\C_{(\alpha)}^{\,v}
\end{array}\]
It is clear that these functions are constant over $\C_{(\alpha)}^{\,v}$, for all $\alpha\in\overline{V}$. So, having all this in mind we obtain
\begin{eqnarray}
\sum_{C\in\Delta}\mu(C)\val(C) & = & \sum_{\alpha\in\overline{V}}\left(\sum_{C\in\C_{(\alpha)}^{\,v}}\mu(C)\val(C)\right)\nonumber\\  & = & \sum_{\alpha\in\overline{V}}\left(\sum_{C\in\C_{(\alpha)}^{\,v}}1\right)\mu(\alpha)\val(\alpha) \nonumber\\ & = & \sum_{\alpha\in\overline{V}}\occ(\alpha,V)\mu(\alpha)\val(\alpha)\label{042}
\end{eqnarray}
As we already know $S_V=v\L/\rad(v\L)$, then  $\ell(v\L)=\ell(\rad(v\L))+1$, thus by (\ref{043}), (\ref{041}) and (\ref{042}) we finally obtain that
\begin{eqnarray*}
\ell(v\L) & = & \sum_{\alpha\in\overline{V}}\occ(\alpha,V)\mu(\alpha)\val(\alpha)-\sum_{\alpha\in\overline{V}}\occ(\alpha,V)+2,\\ & = & 2+\sum_{\alpha\in\overline{V}}\occ(\alpha,V)(\mu(\alpha)\val(\alpha)-1).
\end{eqnarray*}
\end{proof}

\begin{cor}\label{047}
Let $\L=K\Q/I$ be a Brauer configuration algebra associated to the Brauer configuration algebra $\G$, and let $V$ be a polygon of $\G$. If $v$ is the vertex in $\Q$ associated to $V$ then \[\dime_Kv\L=2+\sum_{\alpha\in\overline{V}}\occ(\alpha,V)(\val(\alpha)\mu(\alpha)-1).\]
\end{cor}
\begin{proof}
It follows immediately because $\ell(v\L)=\dim_Kv\L$.
\end{proof} 

One thing that we can observe of the formula in Corollary \ref{047} is its resemblance with the formula in \cite[Proposition 3.3]{mya}.\\

From Corollary \ref{047} we obtain an alternative expression for the vector dimension of a Brauer configuration algebra.

\begin{propo}\label{050}
Let $\L=K\Q/I$ be a Brauer configuration algebra associated to the Brauer configuration $\G=(\G_0,\G_1,\mu,\oo)$. Then \[\dime_K\L=2|\G_1|+\sum_{V\in\G_1}\left(\sum_{\alpha\in\overline{V}}\occ(\alpha,V)(\mu(\alpha)\val(\alpha)-1)\right).\]
\end{propo}

\begin{proof}
It follows from the decomposition $\L\cong\coprod_{v\in\Q_0}v\L=\coprod_{V\in\G_1}v\L$ as a right module.
\end{proof}

To finish this section, we obtain a combinatorial relation that it is satisfied by the objects that form a Brauer configuration by using the Cartan matrix of the previous section and the formula of Corollary \ref{047}.
\begin{propo}\label{051}
Let $\G=(\G_0,\G_1,\mu,\oo)$ be a Brauer configuration. Then \[\sum_{\alpha\in\overline{V}}\occ(\alpha,V)\val(\alpha)\mu(\alpha)=\sum_{W\in\G_1}\left(\sum_{\alpha\in\overline{V}\cap\overline{W}}\mu(\alpha)\occ(\alpha,V)\occ(\alpha,W)\right)\] for each polygon $V$ in $\G$.
\end{propo}
\begin{proof}
Let $\L=K\Q/I$ be the Brauer configuration algebra associated to $\G$ and let $\CM_{\L}$ be its Cartan matrix, and where $K$ is an algebraically closed field. As we know $\CM_{\L}$ is a symmetric matrix, then for $V\in\G_1$ the sum of all entries of the $v$-th row of $\CM_{\L}$ coincides with the $K$-dimension of $v\L$, thus, by Proposition \ref{046}, we have

\begin{eqnarray}
\dime_Kv\L & = & \sum_{W\in\G_1}\dime_Kv\L w\nonumber\\ & = & 2+\sum_{\alpha\in\overline{V}}\occ(\alpha,V)(\occ(\alpha,V)\mu(\alpha)-1)\nonumber\\ & & +\sum_{W\in\G_1\setminus\{V\}}\left(\sum_{\alpha\in\overline{V}\cap\overline{W}}\mu(\alpha)\occ(\alpha,V)\occ(\alpha,W)\right)\label{048}
\end{eqnarray}
Now, it is easy to see that the expression in (\ref{048}) es equal to \[2+\sum_{W\in\G_1}\left(\sum_{\alpha\in\overline{V}\cap\overline{W}}\mu(\alpha)\occ(\alpha,V)\occ(\alpha,W)\right)-\sum_{\alpha\in\overline{V}}\occ(\alpha,V)\] then by Corollary \ref{047} we obtain the result. 
\end{proof}

\section{Brauer configurations induced by a finite group}\label{sec_indu}
In this section we show that given a non-trivial finite group of order different from a prime number, it always induces a class of natural Brauer configurations. As we have seen, two important combinatorial data in a Brauer configuration are the set of vertices and the set of polygons. The other two combinatorial data, the multiplicity function and the orientation of the polygons, can drastically change a Brauer configuration in the sense that either a different multiplicity function or a different orientation defines a new Brauer configuration. However, we saw from these two combinatorial data that the orientation does not influence either the calculus of the Cartan matrix or the length of the indecomposable projective modules of the associated Brauer configuration algebra. In \cite[Proposition 3.13]{brau} we can also see that the vector dimension of a Brauer configuration algebra remains equal if we choose another orientation over the set of polygons, as well in the formula obtained in Proposition \ref{050} that expresses the vector dimension too. In summary, both the vector dimension and the Cartan matrix of a Brauer configuration algebra are invariant with respect to the orientation of the associated Brauer configuration. This particularity will permit us to discover a couple of formulas that are satisfied for all finite groups (see Theorems \ref{057} and \ref{062}). We will get these formulas by defining first the concept of \textit{subgroup-occurrence of an element in a group}, then considering, when it is possible, the Brauer configuration induced by a finite group, and finally applying the formulas expressed in Propositions \ref{050} and \ref{051}. \\

Throughout this section, we use the usual expression $H\leqslant G$ to denote that $H$ is a subgroup of the group $G$, and by $\langle x\rangle$ the subgroup of $G$ generated by the element $x\in G$. If $H\leqslant G$, we also use the symbol $|H|$ to denote the order of $H$ in $G$. For an element $x\in G$ the symbol $|x|$ expresses the order of $x$ in $G$, unless otherwise specified.\\

\begin{defin}
Let $G$ be a group. For $x\in G$, the \textit{subgroup-occurrence of} $x$ \textit{in} $G$ is defined as the value \[\occ_G(x):=|\left\{H\leqslant G\,|\,x\in H\right\}|,\] that is, $\occ_G(x)$ is the cardinal number of the set of all subgroups in $G$ where $x$ belongs. When it is clear from the context, we just say occurrence of an element instead of subgroup-occurrence of an element.
\end{defin}

If $G$ is not a finite group, there could be elements $x\in G$ such that $\occ_G(x)$ is not a finite value. This does not happen when $G$ is a finite group. Independently whether $G$ is finite or not, if $e$ is its identity element, the value $\occ_G(e)$ coincides with the cardinal number of the collection of all subgroups that are contained in $G$.
\begin{ejem}\label{067}
Let us consider the additive group $\Z$ of integers. Using the fact that $\Z$ is a principal ideal domain, we can prove that $\occ_{\Z}(m)<\infty$, for any $m\neq0$. In fact, we have the values\[\occ_{\Z}(m)=\left\{\begin{array}{lr}\tau(|m|), & m\neq0;\\\infty, & m=0,\end{array}\right.\] where $\tau$ is the number of divisors funtion and $|m|$ is the absolute value of $m$.
\end{ejem}
The following proposition presents basic properties that follow from the definition of the subgroup-occurrence of an element in a group.

\begin{propo}\label{056}
 Let $G$ be a group with identity element $e$. Then
 \begin{enumerate}
  \item\label{056a} $G=\left\{e\right\}\iff\occ_G(e)=1$.
  \item\label{056b} $G$ is a cyclic group $\iff\exists x\in G; \occ_G(x)=1$.
  \item\label{056e} $\occ_G(x)=\occ_G(x^{-1})$, for all $x\in G$.
  \item\label{056.1} If the order of $G$ is finite then
  \begin{enumerate}
  \item\label{056c} $\occ_G(e)=2\iff |G|$ is a prime number.
  \item\label{056d} $\occ_G(x)=\occ_G(e)\iff x=e$.
  \end{enumerate}
 \end{enumerate}
\end{propo}
\begin{proof}
 \ref{056a}. It is ovious from the fact that $\occ_G(e)$ coincides with the number of subgroups contained in $G$.
 
 \ref{056b}. Let $x$ be an element of $G$ such that $G=\langle x\rangle$. If $H\leqslant G$ such that $x\in H$, then $G=\langle x\rangle\subset H$, i.e, $G=H$ and therefore $\occ_G(x)=1$. Conversely, if $x\in G$ is an element such that $\occ_G(x)=1$, then necessarily $G=\langle x\rangle$ because $x\in\langle x\rangle$.
 
 \ref{056e}. If $H$ is a subgroup of $G$ such that $x\in H$, we know that \[x\in H\iff x^{-1}\in H.\] So, the collection of all soubgroups containing $x$ coincides with the collection of all subgroups containing $x^{-1}$. The affirmation follows.
 
 \ref{056c}. The implication ``$\Longleftarrow$'' is obvious. Conversely, if that $\occ_G(e)=2$ then we have that the only subgroups of $G$ are $\{e\}$ and $G$, but as we know, this necessarily implies that the order order of $G$ must be a prime number.
 
 \ref{056d}. The implication ``$\Longleftarrow$'' is also obvious. Conversely, if $X$ is the set of all subgroups of $G$ where $x$ belongs then \[X\subset\left\{H\,|\,H\leqslant G\right\}\] but $\occ_G(x)=\occ_G(e)$ then it follows that $X=\left\{H\,|\,H\leqslant G\right\}$, thus $x$ belongs to each subgroup of $G$, in particular $x\in\left\{e\right\}$ and $x=e$.
\end{proof}

To determine the formulas that we mentioned at the beginning of the section, we first consider a class of Brauer configurations induced by a finite group whose order is not a prime number.

\begin{propo}\label{059}
Let $G$ be a finite group with identity element $e$, and such that $|G|>1$ is not a prime number. Let $\G=(\G_0,\G_1,\mu,\oo)$ be the configuration defined by
\begin{enumerate}
 \item $\G_0=G$.
 \item $\G_1=\left\{H\,|\,H\leqslant G\textrm{ and }H\neq\left\{e\right\}\right\}$.
 \item $\mu:G\to\Z_{>0}$ is a function such that $\mu(e)=1$.
 \item $\oo$ is a chosen fixed orientation over the objects in $\G_1$.
\end{enumerate}
Then $\G$ is a Brauer configuration.
\end{propo}
\begin{proof}
 The only thing we need to guarantee is that each polygon in $\G_1$ has at least one non-truncated vertex. In fact, it is clear that $e\in H$ for any $H\in\G_1$. Now, if $p$ is a prime divisor of $|G|$ then by Cauchy's Theorem there exists $x\in G$ such that $|x|=p$. Because of this we also obtain an object $\langle x\rangle\in\G_1$ such that \[\left\{e\right\}\subsetneqq\langle x\rangle\subsetneqq G,\] which implies that $\val(e)\mu(e)=\val(e)\ge2$, thus we can affirm that each polygon in $\G_1$ has at least one non-truncated vertex. The result follows.
\end{proof}

Let $G$ be a finite group with identity element $e$, and such that $|G|>1$ is not a prime number. Let $\G=(\G_0,\G_1,\mu,\oo)$ be an induced Brauer configuration of the type in Proposition \ref{059}. If $\val:\G_0\to\Z_{>0}$ is the valence function associated to the Brauer configuration $\G$, we can easily see that \[\val(x)=\left\{\begin{array}{lr}\occ_G(x), & x\neq e;\\\occ_G(x)-1, & x=e.\end{array}\right.\] Each polygon in $\G_1$ is a set, therefore we have the property  \[\occ(x,H)\le1,\textrm{ for all }x\in G,\] and for each $H\in\G_1$. By applying the formula in Proposition \ref{051} to the polygon $H$ in $\G_1$ we obtain that

\[
 \sum_{x\in H}\mu(x)\val(x) = \sum_{\{e\}\neq L\leqslant G}\left(\sum_{x\in H\cap L}\mu(x)\right),
\]
then by using the expression for the valence given above this equality can be expressed as
\[\sum_{x\in H\setminus\{e\}}\mu(x)\occ_G(x)+\occ_G(e)-1=\sum_{\{e\}\neq L\leqslant G}\left(\sum_{x\in H\cap L}\mu(x)\right),\] which is the same like
\begin{equation}\label{069}
\sum_{x\in H}\mu(x)\occ_G(x)=\sum_{L\leqslant G}\left(\sum_{x\in H\cap L}\mu(x)\right).
\end{equation}

Observe that this equality is also satisfied even when $H=\left\{e\right\}$.

\begin{theo}\label{057}
 Let $G$ be a finite group with identity element $e$ and let $\mu:G\to\Z_{>0}$ be a function such that $\mu(e)=1$. If $H\leqslant G$ then \[\sum_{x\in H}\mu(x)\occ_G(x)=\sum_{L\leqslant G}\left(\sum_{x\in H\cap L}\mu(x)\right).\]
\end{theo}
\begin{proof}
For what we just did previously, we only need to demonstrate the equality when the order of $G$ is a prime number. So, if $|G|$ is a prime number then by Proposition \ref{056} we have

\begin{equation}\label{070}
\begin{array}{rclr}
\occ_G(e) & = & 2, & \\\occ_G(x) & = & 1, &\forall x\in G\setminus\{e\}.
\end{array}
\end{equation}

As we said before, the equality in (\ref{069}) is satisfied for $H=\{e\}$, and due to the fact that $G$ contains only the trivial subgroups, it is enough to check the equality for $H=G$. So, by the equalities in (\ref{070}) it follows that

\begin{eqnarray*}
\sum_{x\in G}\mu(x)\occ_G(x) & = & 2+\sum_{x\in G\setminus\{e\}}\mu(x)\\ & = & 1+\sum_{x\in G}\mu(x),
\end{eqnarray*}
which clearly coincides with (\ref{069}) for the case $H=G$.
\end{proof}

\begin{cor}\label{071}
Let $G$ be a finite group. If $H\leqslant G$ then \[\sum_{x\in H}\occ_G(x)=\sum_{L\leqslant G}|H\cap L|.\]
\end{cor}
\begin{proof}
The result follows by defining $\mu:G\to\Z_{>0}$ as the constant function $\mu\equiv1$.
\end{proof}
Setting $G=\Z_n$ the cyclic group of the integers module $n$, we obtain the following expressions.
\begin{cor}\label{058}
 Let $n$ be a positive integer. If $k$ is a positive divisor of $n$ then \begin{eqnarray*}\sum_{x\in\langle\frac{\ov{n}}{k}\rangle}\occ_{\Z_n}(x) & = & \sum_{d\,|\,n}\gcd(k,d),\\\sum_{x\in\langle\frac{\ov{n}}{k}\rangle}|x|\occ_{\Z_n}(x) & = & \sum_{d\,|\,n}\left(\sum_{t\,|\,\gcd(k,d)}\phi(t)t\right),\end{eqnarray*} where $\phi$ is the Euler phi-function.
\end{cor}
\begin{proof}
 If $k$ is a positive divisor of $n$ we know that $\langle\frac{\ov{n}}{k}\rangle$ is the unique subgroup  of $\Z_n$ of order $k$. On the other hand, if $d$ is another positive divisor of $n$ then $\langle\frac{\ov{n}}{k}\rangle\cap\langle\frac{\ov{n}}{d}\rangle=\langle\frac{\ov{n}}{\gcd(k,d)}\rangle$, thus we obtain
 \begin{eqnarray*}
  \sum_{x\in\langle\frac{\ov{n}}{k}\rangle}\occ_{\Z_n}(x) & = & \sum_{d\,|\,n}\left|\left\langle\frac{\ov{n}}{k}\right\rangle\cap\left\langle\frac{\ov{n}}{d}\right\rangle\right|\\ & = & \sum_{d\,|\,n}\gcd(k,d).
 \end{eqnarray*}

The function $\mu(x)=|x|$, for all $x\in\Z_n$, clearly satisfies the condition $\mu(\ov{0})=1$, then
\begin{eqnarray*}
\sum_{x\in\langle\frac{\ov{n}}{k}\rangle}|x|\occ_{\Z_n}(x) & = & \sum_{d\,|\,n}\left(\sum_{x\in\langle\frac{\ov{n}}{\gcd(k,d)}\rangle}|x|\right).
\end{eqnarray*}
Now, if $t$ is a positive divisor of $\gcd(k,d)$, then the subgroup $\langle\frac{\ov{n}}{\gcd(k,d)}\rangle$ has exactly $\phi(t)$ elements of order $t$, hence  \[\sum_{x\in\langle\frac{\ov{n}}{\gcd(k,d)}\rangle}|x|=\sum_{t\,|\,\gcd(k,d)}\phi(t)t,\] and the result follows.
\end{proof}
In particular, if $k=n$ then we have that
\begin{eqnarray}
\sum_{x\in\Z_n}\occ_{\Z_n}(x) & = & \sigma(n),\label{075}\\ \sum_{x\in\Z_n}|x|\occ_{\Z_n}(x) & = & \sum_{d\,|\,n}\left(\sum_{t\,|\,d}\phi(t)t\right),\nonumber
\end{eqnarray}
where $\sigma$ is the sum of divisors function. The equality in (\ref{075}) implies the following property.

\begin{cor}\label{076}
Let $m$ and $n$ be positive integers such that $\gcd(m,n)=1$. Then  \[\sum_{x\in\Z_{mn}}\occ_{\Z_{mn}}(x)=\left(\sum_{x\in\Z_{m}}\occ_{\Z_{m}}(x)\right)\left(\sum_{x\in\Z_{n}}\occ_{\Z_{n}}(x)\right).\]
\end{cor}
\begin{proof}
Apply the fact that $\sigma$ is a multiplicative function.
\end{proof}
Let us see the following easy example.
\begin{ejem}\label{072}
The cyclic group $\Z_{12}$ has exactly six subgroups, which are given by

\[
\begin{array}{rcl}
\langle\ov{0}\rangle & = & \{\ov{0}\},\\
\langle\ov{6}\rangle & = & \{\ov{0},\ov{6}\},\\
\langle\ov{4}\rangle & = & \{\ov{0},\ov{4},\ov{8}\},
\end{array}
\begin{array}{rcl}
\langle\ov{3}\rangle & = & \{\ov{0},\ov{3},\ov{6},\ov{9}\},\\
\langle\ov{2}\rangle & = & \{\ov{0},\ov{2},\ov{4},\ov{6},\ov{8},\ov{10}\}, \\
\langle\ov{1}\rangle & = & \Z_{12}.
\end{array}
\]
These subgroups can be represented by the subgroup lattice diagram

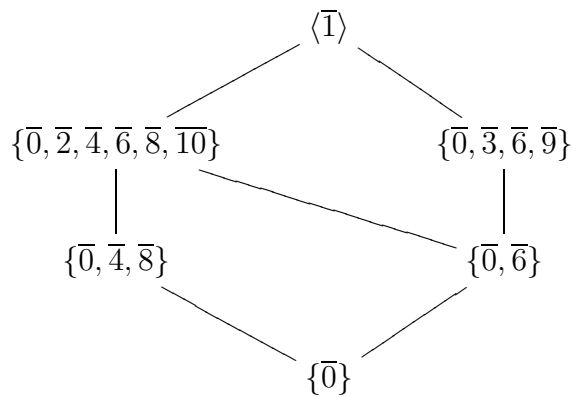
\begin{figure}[H]
\[
\xymatrix{
 &\langle\ov{1}\rangle\ar@{-}[dl]\ar@{-}[dr] &  \\ \{\ov{0},\ov{2},\ov{4},\ov{6},\ov{8},\ov{10}\}\ar@{-}[d]\ar@{-}[rrd] & & \{\ov{0},\ov{3},\ov{6},\ov{9}\}\ar@{-}[d]\\ \{\ov{0},\ov{4},\ov{8}\}\ar@{-}[rd] & & \{\ov{0},\ov{6}\}\ar@{-}[ld]\\ &\{\ov{0}\} & 
}
\]
\caption{Subgroup lattice of $\Z_{12}$.}
\end{figure}
By using this graphic representation, it can be calculated the following table of occurrences of elements in $\Z_{12}$.
\begin{table}[H]
\begin{center}
  \begin{tabular}{C||C|C|C|C|C|C|C|C|C|C|C|C}
x & \ov{0} & \ov{1} & \ov{2} & \ov{3} & \ov{4} & \ov{5} & \ov{6} & \ov{7} & \ov{8} & \ov{9} & \ov{10} & \ov{11} \\ \hline \occ_{\Z_{12}}(x) & 6 & 1 & 2 & 2 & 3 & 1 & 4 & 1 & 3 & 2 & 2 & 1
  \end{tabular}
\end{center}
\caption{Table of occurrences of $\Z_{12}$.}
\end{table}
Now, considering the subgroup $\langle\frac{\ov{12}}{4}\rangle=\langle\ov{3}\rangle=\{\ov{0},\ov{3},\ov{6},\ov{9}\}$ we see that the sum of the occurrences of  the elements in this group is \begin{eqnarray*}\occ_{\Z_{12}}(\ov{0})+\occ_{\Z_{12}}(\ov{3})+\occ_{\Z_{12}}(\ov{6})+\occ_{\Z_{12}}(\ov{9}) & = & 6+2+4+2\\ & = & 14,\end{eqnarray*} and the sum of the greatest common divisors is
{\small
\begin{multline*}
\gcd(4,1)+\gcd(4,2)+\gcd(4,3)+\gcd(4,4)+\gcd(4,6)+\gcd(4,12) =\\ 1+2+1+4+2+4 = 14.
\end{multline*}
} It is an easy exercise to check the respective equalities for the other subgroups.
\end{ejem}

Once again, let $G$ be a finite group with identity element $e$ and such that $|G|>1$ is not a prime number. Let $\G$ be an induced Brauer configuration of the type in Proposition \ref{059}. From the way that $\G$ is defined, it is clear that \[|\G_1|=\val(e)=\occ_G(e)-1.\] Now, let $\L$ be the Brauer configuration algebra associated to the Brauer configuration $\G$. By \cite[Proposition 3.13]{brau} the vector dimension of $\L$ over the field $K$ can be calculated by the expression
\begin{equation*}
\dim_K\L=2|\G_1|+\sum_{x\in\G_0}\val(x)(\mu(x)\val(x)-1).
\end{equation*}
Replacing the equivalent expressions of valence that we know and that of $|\G_1|$, we obtain that the vector dimension of $\L$ is equal to
{\small
\begin{multline*}
 2(\occ_G(e)-1)+\sum_{x\in G\setminus\left\{e\right\}}\occ_G(x)(\mu(x)\occ_G(x)-1)+(\occ_G(e)-1)(\occ_G(e)-2)=\\\occ_G(e)(\occ_G(e)-1)+\sum_{x\in G\setminus\left\{e\right\}}\occ_G(x)(\mu(x)\occ_G(x)-1)=\\\sum_{x\in G}\occ_G(x)(\mu(x)\occ_G(x)-1),
\end{multline*}
}
that is,
\begin{equation}\label{060}
 \dim_k\L=\sum_{x\in G}\occ_G(x)(\mu(x)\occ_G(x)-1).
\end{equation}
On the other way, by Proposition \ref{050} we have that the vector dimension of $\L$ can also be calculated as

{\small
\begin{multline*}
 2(\occ_G(e)-1)+\sum_{\left\{e\right\}\neq H\leqslant G}\left(\sum_{x\in H}(\mu(x)\val(x)-1)\right)=\\2(\occ_G(e)-1)+\sum_{\left\{e\right\}\neq H\leqslant G}\left(\sum_{x\in H\setminus\{e\}}(\mu(x)\occ_G(x)-1)+\occ_G(e)-2\right)=\\2(\occ_G(e)-1)+\sum_{\left\{e\right\}\neq H\leqslant G}\left(\sum_{x\in H}(\mu(x)\occ_G(x)-1)-1\right)=\\2(\occ_G(e)-1)+\sum_{\left\{e\right\}\neq H\leqslant G}\left(\sum_{x\in H}(\mu(x)\occ_G(x)-1)\right)-(\occ_G(e)-1)=\\ \sum_{ H\leqslant G}\left(\sum_{x\in H}(\mu(x)\occ_G(x)-1)\right),
\end{multline*}
}
that is,
\begin{equation}\label{061}
 \dim_K\L=\sum_{ H\leqslant G}\left(\sum_{x\in H}(\mu(x)\occ_G(x)-1)\right),
\end{equation}
so, by the expression in (\ref{060}) we have that \[\sum_{x\in G}\occ_G(x)(\mu(x)\occ_G(x)-1)=\sum_{ H\leqslant G}\left(\sum_{x\in H}(\mu(x)\occ_G(x)-1)\right).\] By distributing the respective summation symbols in this equality and applying Corollary \ref{071}, we obtain \begin{equation}\label{063}\sum_{x\in G}\mu(x)\occ_G(x)^2=\sum_{H\leqslant G}\left(\sum_{x\in H}\mu(x)\occ_G(x)\right).\end{equation}
\begin{theo}\label{062}
 Let $G$ be a finite group with identity element $e$ and let $\mu:G\to\Z_{>0}$ be a function such that $\mu(e)=1$. Then \begin{eqnarray*}\sum_{x\in G}\mu(x)\occ_G(x)^2 & = & \sum_{H\leqslant G}\left(\sum_{x\in H}\mu(x)\occ_G(x)\right)\\ & = & \sum_{H\leqslant G}\left(\sum_{L\leqslant G}\left(\sum_{x\in H\cap L}\mu(x)\right)\right).\end{eqnarray*}
\end{theo}
\begin{proof}
 The equality is trivial if $G=\{e\}$. Once again, we just need to demonstrate the equality when the order of $G$ is a prime number. In fact, if $|G|$ is a prime number we know that \[\begin{array}{rclr}\occ_G(e) & = & 2, & \\\occ_G(x) & = & 1, &\forall x\in G\setminus\{e\},\end{array}\] or equivalently \[\begin{array}{rclr}\occ_G(e)^2 & = & 2\,\occ_G(e), & \\\occ_G(x)^2 & = & 1, &\forall x\in G\setminus\{e\}.\end{array}\]
 So, we have the equalities
\begin{eqnarray*}
 \occ_G(e)^2+\sum_{x\in G\setminus\{e\}}\mu(x) & = & 2\,\occ_G(e)+\sum_{x\in G\setminus\{e\}}\mu(x)\\ & = & \occ_G(e)+\sum_{x\in G}\mu(x)\occ_G(x), 
\end{eqnarray*} and due to the fact that $G$ contains only the trivial subgroups, this equality coincides with the one in (\ref{063}) and the result follows.
 \end{proof}
\begin{cor}\label{064}
 Let $G$ be a finite group. Then 
 
 \[\sum_{x\in G}\occ_G(x)^2=\sum_{H\leqslant G}\left(\sum_{L\leqslant G}|H\cap L|\right).\]
\end{cor}
\begin{proof}
It follows by defining $\mu\equiv1$.
\end{proof}
By the same reasoning used in the proof of Corollary \ref{058}, we can easily prove the following corollary.
\begin{cor}\label{065}
 Let $n$ be a positive integer. Then \begin{eqnarray*}\sum_{x\in\Z_n}\occ_{\Z_n}(x)^2 & = & \sum_{k\,|\,n}\left(\sum_{d\,|\,n}\gcd(k,d)\right),\\\sum_{x\in\Z_n}|x|\occ_{\Z_n}(x)^2 & = & \sum_{k\,|\,n}\left(\sum_{d\,|\,n}\left(\sum_{t\,|\,\gcd(k,d)}\phi(t)t\right)\right).\end{eqnarray*}
\end{cor}

\section{The dimension of the center of a Brauer configuration algebra induced by a finite group}\label{sec_center}
In this section we consider pairs $(G,\mu)$, where $G$ is a nontrivial finite group with identity element $e$, and $\mu:G\to\Z_{>0}$ is a function such that $\mu(e)=1$. We will associate to a pair $(G,\mu)$ of this type a class $\GG_{(G,\mu)}$ formed by a finite collection of configurations that depending on both the order of $G$ and the multiplicity function $\mu$, the collection $\GG_{(G,\mu)}$ would be formed by Brauer configurations or by an only configuration which is not a Brauer configuration. Then, for each $\G\in\GG_{(G,\mu)}$ we will associate a $K$-algebra $\L_{\G}$. We finalize  the section by showing that when the pair $(G,\mu)$ has constant function $\mu\equiv1$, then the vector dimension of $Z(\L_{\G})$ coincides with the number of subgroups of $G$, and where $\L_{\G}$ is the $K$-algebra associated to $\G\in\GG_{(G,\mu)}$.

But before all that, we start with the following theorem about the vector dimension of the center of certain type of Brauer configuration algebras.
\begin{theo}\label{074}
Let $\G$ be a Brauer configuration with associated Brauer configuration algebra $\L=K\Q/I$, and such that $\G$ is connected and each of its polygons is a set. Then \[\dim_KZ(\L)=1+\sum_{\alpha\in\G_0}\mu(\alpha)+|\G_1|-|\G_0|.\]
\end{theo}
\begin{proof}
If $\Q$ is the induced quiver by the Brauer configuration $\G$, by the proof of \cite[Theorem 4.9]{mya} we have that there are two types of loops in $\Q$. Those induced by the set of vertices $\C_{\G}=\{\gamma\in\G_0\,|\,\val(\gamma)=1\textrm{ and }\mu(\gamma)>1\}$, and those induced by some self-folded polygons. But there are no self-folded polygons in $\G$ because all of them are sets. So, if $\Q$ has any loops, then these loops are induced by the vertices in $\C_{\G}$. Hence, we obtain that $\#\mathit{Loops}(\Q)=|\C_{\G}|$ and the result follows.
\end{proof}
Let $G$ be a non-trivial finite group with identity element $e$, and let $\mu:G\to\Z_{>0}$ a function such that $\mu(e)=1$. If $|G|$ is not a prime number, we denote by $\GG_{(G,\mu)}$ the class\footnote{As we know, given a Brauer configuration $\G$ we can define different orientations $\oo$ over the list of polygons of each vertex in $\G_0$, and each of these orientations gives a new Brauer configuration.} of all the Brauer configurations as constructed in Proposition \ref{059}. The class $\GG_{(G,\mu)}$ is indexed by all the possible orientations $\oo$ over the list of polygons of each nontruncated vertex.
\begin{propo}\label{073}
Let $G$ be a finite group with identity element $e$, and let $\mu:G\to\Z_{>0}$ be a funtion such that $\mu(e)=1$. Suposse that $|G|>1$ is not a prime number. For $\G$ and $\G'$ Brauer configurations of $\GG_{(G,\mu)}$, let $\L_{\G}$ and $\L_{\G'}$ denote the Brauer configuration algebras associated to $\G$ and $\G'$, respectively. Then $\dim_KZ(\L_{\G})=\dim_KZ(\L_{\G'})$.
\end{propo}

\begin{proof}
It's an immediate consequence of Theorem \ref{074}.
\end{proof}

Now, let $(G,\mu)$ be a pair such that $|G|$ is a prime number, and let $\mu:G\to\Z_{>0}$ be a function such that $\mu(e)=1$ but $\mu\not\equiv1$. Applying the same construction that appears in Proposition \ref{059}, let $\G=(\G_0,\G_1,\mu,\oo)$ be the configuration formed by
\begin{itemize}
\item $\G_0=G$;
\item $\G_1=\{G\}$ and
\item $\oo$ is the list formed by each nontruncated vertex that has as list of polygons the one formed only by $G$. 
\end{itemize}
Due to the fact that $\mu\not\equiv1$, necessarily the only polygon of $\G_1$ has at least one nontruncated vertex, hence $\G$ is a Brauer configuration. In this case, the collection $\GG_{(G,\mu)}$ is formed only by the configuration $\G$.

Now, let us consider an extreme case. If $(G,\mu)$ is a pair such that $|G|$ is a prime number and $\mu:G\to\Z_{>0}$ is the constant function $\mu\equiv1$, then the induced configuration as constructed in Proposition \ref{059} has the property  that every vertex in it is truncated. That is, if $\G=(\G_0,\G_1,\mu,\oo)$ is the configuration induced by $(G,\mu)$, where  $\G_0=G,\G_1=\{G\}$, then \[\val(x)=\mu(x)=1,\textrm{ for any }x\in G.\] According to \cite[Definition 1.5]{brau}, condition C3 is not satisfied, thus the configuration $\G$ cannot be a Brauer configuration. However, in this  particular case, we still associate to the pair $(G,\mu)$ the collection $\GG_{(G,\mu)}$ formed only by $\G$.

In summary, we associate to the given pair $(G,\mu)$ the collection $\GG_{(G,\mu)}$ formed by all the Brauer configurations as constructed in Proposition \ref{059}, if $|G|$ is not a prime number, or if $|G|$ is a prime number but $\mu\not\equiv1$. For each $\G\in\GG_{(G,\mu)}$ we also associate the induced quiver $\Q_{\G}$ and the induced Brauer configuration algebra $\L_{\G}$. If $|G|$ is a prime number and $\mu\equiv1$, then $\GG_{(G,\mu)}$ is formed by the only configuration $\G$ as we just constructed above for this case; to this configuration $\G$ we associate the quiver $\Q_{\G}=\xymatrix{\cdot\ar@(dr,ur)_{x}}$ and the Brauer graph algebra $\L_{\G}=K[x]/(x^2)$ (see \cite[Page 540]{extal}).

\begin{theo}\label{066}
Let $(G,\mu)$ be a pair where $G$ is a nontrivial finite group with identity element $e$, and $\mu:G\to\Z_{>0}$ is the constant funtion $\mu\equiv1$. If $\G\in\GG_{(G,\mu)}$, then \[\dim_KZ(\L_{\G})=\occ_{G}(e).\]
\end{theo}

\begin{proof}
 First, let us suppose that $|G|$ is not a prime number, and let $\L_{\G}$ be the associated Brauer configuration algebra to the Brauer configuration $\G\in\GG_{(G,\mu)}$. By Theorem \ref{074} we have that the vector dimension of the center of $\L_{\G}$ is equal to
 \begin{eqnarray*}
 \dim_KZ(\L_{\G}) & = & 1+\sum_{x\in G}1+|\G_1|-|G|\\ & = & 1+|G|+\occ_G(e)-1-|G|\\ & = & \occ_G(e).
 \end{eqnarray*}
 Now, if $|G|$ is a prime number, the associated algebra to the unique configuration $\G\in\GG_{(G,\mu)}$ is $\L_{\G}=K[x]/(x^2)$, then $Z(\L_{\G})=K[x]/(x^2)$ and \[\dim_KZ(\L_{\G})=2,\] and clearly the number of subgroups in $G$ is 2. So, either case, the vector dimension of the center of the associated $K$-algebra $\L_{\G}$ coincides with the number of subgroups in $G$.
\end{proof}

\nocite{*}
\bibliographystyle{plain}
\bibliography{Bcabibli}

\end{document}